\documentclass[11pt]{article}
\usepackage{graphicx}
\usepackage{subfigure}
\usepackage{epsfig}
\usepackage{amsmath}
\usepackage{amsthm}
\usepackage{amssymb}
\usepackage{lscape}
\usepackage{color}
\usepackage{tikz}
\usepackage{lineno}
\usepackage[ruled,boxed]{algorithm2e}
\usepackage{booktabs}
\usepackage{rotating}

\newtheorem{thm}{Theorem}[section]

\theoremstyle{definition}

\newtheorem{rem}{Remark}[section]

\begin{document}

\title{Efficient Time Discretization for Exploring Spatial Superconvergence of Discontinuous Galerkin Methods}%
\author{Xiaozhou Li\footnotemark[1]~\footnotemark[2]}

\renewcommand{\thefootnote}{\fnsymbol{footnote}}
\footnotetext[1]{School of Mathematical Sciences, University of Electronic Science and Technology of China, Chengdu, China. (xiaozhouli@uestc.edu.cn)}
\footnotetext[2]{Supported by NSFSC Grant 2022NSFSC1849 and NSFC Grant 11801062}
\footnotetext[3]{Corresponding author.}
\date{}

\maketitle

\begin{abstract}
We investigate two efficient time discretizations for the post-processing technique of discontinuous Galerkin (DG) methods to solve hyperbolic conservation laws. The post-processing technique, which is applied at the final time of the DG method, can enhance the accuracy of the original DG solution (spatial superconvergence). One main difficulty of the post-processing technique is that the spatial superconvergence after post-processing needs to be matched with proper temporary accuracy. If the semi-discretized system (ODE system after spatial discretization) is under-resolved in time, then the space superconvergence will be concealed.  In this paper, we focus our investigation on the recently designed SDG method and derive its explicit scheme from a correction process based on the DG weak formulation. We also introduce another similar technique, namely the spectral deferred correction (SDC) method. A comparison is made among both proposed time discretization techniques with the standard third-order Runge-Kutta method through several numerical examples, to conclude that both the SDG and SDC methods are efficient time discretization techniques for exploiting the spatial superconvergence of the DG methods.           

\end{abstract}

\section{Introduction}
In this paper, we investigate efficient time discretizations for exploring the spatial superconvergence of DG methods for solving hyperbolic problems. 

In the past decades, the DG method has become more and more popular for solving partial differential equations, such as hyperbolic equations, convection-diffusion equations, etc. The first introduction to the DG method was given by Reed and Hill~\cite{Reed:1973} in 1973. Later, this method was extended to the so-called RKDG method, which couples the spatial DG discretization with explicit total variation diminishing (TVD) Runge-Kutta methods, by Cockburn, et al.~\cite{Cockburn:1991,Cockburn:1989a,Cockburn:1989b,Cockburn:1990}. The stability and error estimates for the fully discrete RKDG method can be found in~\cite{Zhang:2004,Zhang:2010}.   

With the development of the DG method, the study of its superconvergence has become an area of increasing interest, since the higher-order information can be extracted from DG solutions by means of the post-processing technique, see~\cite{Cockburn:2003}. This post-processing technique was originally introduced by Bramble and Schatz~\cite{Bramble:1977} in 1977 for enhancing the accuracy of solutions of finite element methods. The first extension of this post-processing technique to DG methods was given by Cockburn et al.~\cite{Cockburn:2003}.  In an ideal situation, the post-processing technique can increase the accuracy order of DG solutions (with polynomials of degree $p$) from $p+1$ to $2p+1$ in the $L^2$ norm. The superconvergence of order $2p+1$ is attractive, which provides an efficient way to obtain high-order accuracy in space since the DG method (without post-processing) requires using polynomials of degree $2p$ to get the same accuracy order.  In recent years, to apply the post-processing technique in practice, researchers have investigated various aspects related to this technique, such as boundary conditions~\cite{Ryan:2003,Ryan:2015}, derivatives~\cite{Ryan:2009,Li:2016}, nonlinear equations~\cite{Ji:2013,Meng:2016}, etc.  
 
However, until now, no work has considered the time discretization for the post-processing technique. Most literature, such as~\cite{Cockburn:2003,Ryan:2003,Ryan:2015,Li:2016}, had simply used an explicit third-order Runge-Kutta method, but the thing is not as easy as it seems to be.  As mentioned before, after the post-processing, the DG approximation will have superconvergence order of $2p+1$ in space. It is an attractive property, but it is also the main difficulty since the temporary accuracy should more or less match the spatial accuracy. A first numerical result has suggested that if the time discretization is under-resolved, then the spatial superconvergence will be concealed, see~\cite{Li:2017}. In this paper, we first demonstrate that for the high-order case ($p \geq 2$), to benefit from the spatial superconvergence, explicit Runge-Kutta methods will suffer from an extremely small time step. This extremely small time step causes the entire algorithm to become inefficient, contrary to the intent of the development of the post-processing technique. To address this issue, we have to look for efficient time discretizations.  In this paper, we define the ``efficient time discretization'' in the sense that it allows using the time step size only restricted by the stability (linear or nonlinear) requirement ($\Delta t = \mathcal{O}(\Delta x)$).  

This paper focuses on the SDG method recently constructed by Li, Benedusi, and Krause~\cite{Li:2017}. The SDG method is an iterative method for time discretization derived from the DG method, with the order of accuracy increased by one for each additional iteration. The general advantage of the SDG method is that it is a one-step method and can be constructed systematically for arbitrary order of accuracy. In particular, the SDG method provides a temporary superconvergence in time, which appropriately matches the spatial superconvergence after post-processing. Also, we give a detailed derivation of the explicit SDG scheme from a correction method based on the DG weak formulation. Another time integrator, the spectral deferred correction method~\cite{Dutt:2000} is also briefly explored in this paper. Numerical experiments are performed to compare these two techniques together with the standard Runge-Kutta methods. The results conclude that both the SDG and SDC methods are efficient when coupled with the DG spatial discretization and the post-processing technique, and the enhancement of the efficiency is significant compared to Runge-Kutta methods. Finally, we remark that both SDG and SDC methods have other attractive properties, such as the flexibility for $hp$ adaptive and the suitable structure for parallel-in-time algorithms, which can be used to enhance the efficiency further. 
 
The paper is organized as follows. In Section 2, we review the DG method coupled with the standard Runge-Kutta method. In Section 3, we describe the post-processing of the DG method and the challenges for efficient time discretizations. In Section 4, we give a description of the SDG method and a brief introduction of the spectral deferred correction (SDC) method. Numerical examples are presented in Section 5. Finally, we give the conclusion in Section 6.      

\section{The Discontinuous Galerkin Method}
In this section, we give a very brief description of the essential concepts of the discontinuous Galerkin method.  We refer to~\cite{Cockburn:1991,Cockburn:1989a,Cockburn:1989b,Cockburn:1990} for details.  

Consider the following one-dimensional hyperbolic equation
\begin{equation}
  \label{eq:hyperbolic}
  u_t + {f(u)}_x = 0, \qquad (x,t) \in \Omega\times[0,T],
\end{equation}
where the spatial domain $\Omega = [a, b] \subset \mathbb{R}$ and $f \in \mathcal{C}^\infty(\mathbb{R})$.  First, we give a partition of $\Omega$ that $a = x_{\frac{1}{2}} < x_{\frac{3}{2}} < \cdots < x_{N+\frac{1}{2}} = b$, and denote elements $I_j = [x_{j-\frac{1}{2}}, x_{j+\frac{1}{2}}]$ with size $\Delta x_j = x_{j+\frac{1}{2}} -  x_{j-\frac{1}{2}}$ for $j = 1,\dots,N$.  For convenience, in this paper, we limit the discussion to the uniform partition (it does not need to be) that $\Delta x_j = \Delta x = \frac{1}{N}(b-a)$ for $j = 1,\dots,N$.   Associated with the partition, we define the DG approximation space
\begin{equation}
  \label{eq:space-s}
  V_h  = \left\{v_h:\, v_h\big\vert_{I_j} \in \mathbb{P}^p(I_j),\,j = 1, \dots, N\right\},
\end{equation}
where $\mathbb{P}^p(I_j)$ denotes the set of polynomials of degree up to $p$ defined on the element $I_j$.  Clearly, $V_h$ is a piecewise polynomial space, and functions in $V_h$ may have discontinuities across element interfaces.  Therefore, we introduce the notation $u_h(x_{j-\frac{1}{2}}^{-})$ and $u_h(x_{j-\frac{1}{2}}^{+})$ for $u_h \in V_h$ i.e.\ the value of $u_h$ at $x_{j-\frac{1}{2}}$ from the left side and from the right side, respectively.   

Then the DG approximation, $u_h$, of the hyperbolic equation~\eqref{eq:hyperbolic} satisfies the semi-discretized weak formulation  
\begin{equation}
  \label{eq:weak-s}
  \int_{I_j} {(u_h)}_t v_h \,dx - \int_{I_j} f(u_h){(v_h)}_x \,dx + \hat{f}v_h\vert_{x_{j+\frac{1}{2}}^{-}} - \hat{f}v_h\vert_{x_{j-\frac{1}{2}}^{+}} = 0, \quad \forall v_h \in V_h,
\end{equation}
for all $1 \leq j \leq N$.  We note that the ``hat'' terms $\hat{f} = \hat{f}(u_h^{-},u_h^{+})$ are referred to as the numerical flux, which is essential to ensure the stability of the numerical scheme.  In this paper, we simply choose the well known Lax-Friedrich flux 
\[
  \hat{f} = \hat{f}(u_h^{-},u_h^{+}) = \frac{1}{2}\left(f(u_h^{-}) + f(u_h^{+}) - \alpha(u_h^{+} - u_h^{-})\right), \qquad \alpha = \max\limits_{u_h}|f'(u_h)|.
\]

To solve equation~\eqref{eq:hyperbolic}, we still need a time discretization for solving the semi-discretized system (ODEs)
\begin{equation}
  \label{eq:system}
  \frac{d }{dt}u_h = \mathcal{L}_h(t, u_h),
\end{equation}
arising from the spatial discretization.

One of the most popular time discretizations for the DG method is the total variation diminishing (TVD) third-order Runge-Kutta method, which is also referred to as one of the strong stability preserving (SSP) methods, see~\cite{Gottlieb:2009}. The TVD third-order Runge-Kutta method we consider is the one introduced in~\cite{Shu:1988},  
\begin{equation}
  \label{eq:RK3}
  \begin{split}
    u_h^{(1)} & = u^n + \Delta t \mathcal{L}_h(t_n, u_h^n), \\
    u_h^{(2)} & = \frac{3}{4}u_h^n + \frac{1}{4}\left(u_h^{(1)} + \Delta t \mathcal{L}_h(t_n + \Delta t, u_h^{(1)})\right), \\
    u_h^{n+1} & = \frac{1}{3}u_h^n + \frac{2}{3}\left(u_h^{(2)} + \Delta t \mathcal{L}_h(t_n + \frac{1}{2}\Delta t, u_h^{(2)})\right), 
  \end{split}
\end{equation}
for computing the approximation, $u_h^{n+1}$ at time $t_{n+1}$ ($t_n +\Delta t$). This TVD third Runge-Kutta method has excellent stability property, i.e., for hyperbolic equations~\cite{Shu:1988,Gottlieb:1998}.  It is moreover well known to show excellent performance when combined with the DG method, which has been verified numerically by different authors~\cite{Shu:1988, Cockburn:2001, Gottlieb:1998}.  We point out that this so-called Runge-Kutta/Discontinuous Galerkin Method (RKDG) has already become a classical and popular numerical scheme for solving conservation laws.    

\section{Post-processing of DG Method}
With the development of the DG method, the study of its superconvergence has become an area of increasing interest, since the higher-order information can be extracted easily from DG solutions by means of post-processing techniques.  In this paper, we only consider the classical post-processing technique that stems from the works of Bramble and Schatz~\cite{Bramble:1977}.  An extension to this technique to discontinuous Galerkin methods was first introduced by Cockburn et al.~\cite{Cockburn:2003} in 2003. Since then, various aspects of post-processing techniques have been studied in the literature, see~\cite{Ryan:2009,Ji:2013,Ryan:2015,Li:2015,Li:2016}. 

The post-processing is applied only at the final time $T$ of the DG approximation.  The post-processed solution, $u_h^{\star}$, is given by 
\begin{equation}
  \label{eq:ustar}
    u_h^{\star}(x, T) = \left(K_h^{(2p+1,\,p+1)}\ast u_h\right)(x, T)= \int_{-\infty}^{\infty}K_h^{(2p+1,\,p+1)}(x - \xi)u_h(\xi, T)\, d\xi,
\end{equation}
where the kernel, $K^{(2p+1,\,p+1)}$, is a linear combination of central B-splines,
\begin{equation}
    \label{eq:kerner}
    K^{(2p+1,\,p+1)}(x) = \sum\limits_{\gamma=0}^{2p}c^{(2p+1,\,p+1)}_{\gamma}\psi^{(p+1)}\left(x - (-p+\gamma)\right),
\end{equation}
and the scaled kernel 
\[
  K_h^{(2p+1,\,p+1)}(x) = \frac{1}{h}K^{(2p+1,\,p+1)}\left(\frac{x}{h}\right)
\] 
with the scaling $h = \Delta x$.  The $p+1$ order central B-spline,$\,\psi^{(p+1)}(x)$, can be constructed recursively for $\ell \geq 1$ 
\begin{equation}
    \begin{split}
      \psi^{(1)} & = \chi_{[-\frac{1}{2},\frac{1}{2}]}(x), \\
        \psi^{(\ell+1)}(x) & =  \frac{1}{\ell}\left(\left(\frac{\ell+1}{2}+x\right)\psi^{(\ell)}\left(x+\frac{1}{2}\right) + \left(\frac{\ell+1}{2}-x\right)\psi^{(\ell)}\left(x-\frac{1}{2}\right)\right).
    \end{split}
\end{equation}
The coefficients $c^{(2p+1,p+1)}_{\gamma}$ are computed by requiring the property that the kernel reproduces polynomials by convolution up to degree $2p$, 
\[
  K^{(2p+1,p+1)} \ast r = r, \quad r = 1, x,\dots, x^{2p}.
\]

This post-processing technique is designed to enhance the accuracy order of DG solutions in the $L^2$ norm.  The error analysis has been carried out for linear hyperbolic equations with periodic boundary conditions and can be found in~\cite{Cockburn:2003}.   
\begin{thm}[Cockburn et al.~\cite{Cockburn:2003}]\label{thm:pp}
Let $u$ be the exact solution for the linear equation~\eqref{eq:hyperbolic}, $f(u) = a u$ ($a$ is a constant), with periodic boundary conditions, and $u_h$ the DG approximation derived by scheme~\eqref{eq:weak-s} over a uniform grid.  Then, the post-processed solution satisfies 
\[ 
    \|u - K_h^{(2p+1,p+1)}\ast u_h\|_{L^2} \leq C h^{2p+1}.
\]
\end{thm}

\subsection{Runge-Kutta Methods}
\label{section-rk}
As shown in Theorem~\ref{thm:pp}, the post-processing technique can significantly enhance the accuracy order of DG solutions from $p+1$ up to $2p+1$.  In other words, it provides an efficient way to obtain high-order approximations.  However, to the knowledge of the authors, the time discretization for the semi-discretized system~\eqref{eq:system} has never been investigated for the post-processing technique.  More precisely, most literature had simply applied the third Runge-Kutta method~\eqref{eq:RK3} with a note that said a ``sufficient'' small time step is used such that the error is dominated by spatial errors.  It works fine to verify the convergence rate of the post-processing technique proved in Theorem~\ref{thm:pp}.  However, for practical usage, this ``sufficient'' small time step sometimes leads to significant inefficiency.   

To demonstrate the problem, consider the linear hyperbolic equation
\begin{equation}
  \label{eq:ex1}
  \begin{split}
    u_t + u_x & = 0, \quad (x,t) \in [0,1]\times [0, T]\\
       u(x,0) & =  \sin(2\pi x)
  \end{split}
\end{equation}
with final time $T = 1$.  For uniform grids, we give the $L^2$ norm error and respecting convergence order are in Table~\ref{table:ex1}.  The results show that the post-processed solutions have the superconvergence order $2p+1$.  For the time step size, in this example, we use $\Delta t= \text{cfl} \Delta x$ with $\text{cfl} = 0.1$ for $p = 1$, and $\text{cfl} = 0.01$ for $p = 2$.  We note that the CFL number is reasonable for $\mathbb{P}^1$ case, but it is quite small for $\mathbb{P}^2$ compared to standard usage.  To further demonstrate this issue, we use a high-order approximation with polynomials of degree $4$ and $160$ elements for the DG discretization, and compare the results when different CFL numbers are used, see Figure~\ref{fig:ex1cfl}.  Here, for the post-processed solution,  we can see that a sufficiently small CFL number (less than $10^{-4}$) is way too small.  In this example, using the post-processing technique requires a time step size, which is $100$ times smaller than when using the standard DG method (without post-processing).  As a consequence, the efficiency provided by the post-processing technique is ruined by the extremely small time step size!  

\begin{table}\centering 
\begin{tabular}{@{}cccccccc@{}}\toprule
  & & & \multicolumn{2}{c}{DG error} & \phantom{abc}& \multicolumn{2}{c}{After post-processing} \\ 
\cmidrule{4-5} \cmidrule{7-8}
Degree & $N$ & & $L^2$ error & Order && $L^2$ error & Order \\ 
\midrule
$p = 1$ & 20  && 4.60E-03 &   --  && 1.97E-03 &  -- \\ 
        & 40  && 1.09E-03 &  2.08 && 2.44E-04 & 3.02\\ 
        & 80  && 2.67E-04 &  2.02 && 3.02E-05 & 3.01\\ 
        & 160 && 6.65E-05 &  2.01 && 3.76E-06 & 3.01\\ 
\midrule
$p = 2$ & 20  && 1.07E-04 &   --  && 4.11E-06 &   -- \\
        & 40  && 1.34E-05 &  3.00 && 9.49E-08 &  5.44\\
        & 80  && 1.67E-06 &  3.00 && 2.49E-09 &  5.25\\
        & 160 && 2.09E-07 &  3.00 && 7.75E-11 &  5.00\\
\bottomrule
\end{tabular}
\caption{$L^2-$ errors of the DG approximation $u_h$ and the post-processed solution $u^\star_h$ for the linear hyperbolic equation~\eqref{eq:ex1}, see Section~\ref{section-rk}.}
\label{table:ex1}
\end{table}


\begin{figure}[!ht]z
    \centering
    \includegraphics[width=.7\textwidth]{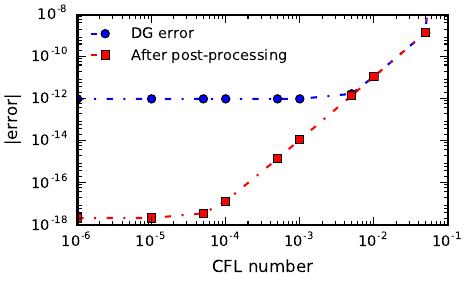}
    
    \begin{center}
    \caption{\label{fig:ex1cfl}
        The DG and post-processed errors when using the third order Runge-Kutta method with different CFL numbers for the linear hyperbolic equation~\eqref{eq:ex1}.  A sufficiently small CFL number means the error has reached the horizontal position, the DG error (blue) and the post-processed error (red), which the error is dominated by the spatial error. 
}
     \end{center}
\end{figure}

\subsection{Challenges of Time Discretization}
From the previous example, we see that one has to use an extremely small time step size for the third-order Runge-Kutta method.  Otherwise, the error of the post-processed approximation will be dominated by the temporary error instead of the desired spatial error.  Generally, when solving the semi-discretized system~\eqref{eq:system}, the time step size needs to satisfy
\begin{equation}
  \label{eq:t-cond}
  \Delta t = \min\{\Delta t_{stable}, \Delta t_{acc}\},
\end{equation}
where $\Delta t_{stable}$ represents the restriction of the stability (linear/nonlinear) requirement, and $\Delta t_{acc}$ represents the restriction of the accuracy requirement.  For the standard DG method (without considering its post-processing), $\Delta t$ is usually dominated by the stability requirement $\Delta t_{stable}$ when the Runge-Kutta methods are used to solve~\eqref{eq:system}.  It leads to one motivation of studying of SSP methods: to improve the limitation of $\Delta t_{stable}$.  However, once the post-processing technique is applied, the situation changes, the accuracy restriction $\Delta t_{acc}$ becomes the limitation.  The reason is straightforward, the accuracy order of the time discretization does not match with the spatial discretization.  The Runge-Kutta method~\eqref{eq:RK3} has convergence rate of $3$ in time, but the post-processed solution has the superconvergence order of $2p+1$ in space.  Clearly, to obtain the same level of accuracy in both space and time, the time step size has to satisfy the relation that  
\begin{equation}
  \label{eq:order}
  {{\Delta t}_{acc}}^3 = \mathcal{O}({\Delta x}^{2p+1}) \quad \Longrightarrow \quad \Delta t_{acc} = \mathcal{O}\left({\Delta x}^{\frac{2p+1}{3}}\right).
\end{equation}
Under the limitation of~\eqref{eq:order}, we can see that once $p \geq 3$, $\Delta t_{acc}$ will have to take a high order form of ${\Delta x}$ ($\Delta t_{stable}$ is only the first order of $\Delta x$), and this means the proper time step size $\Delta t \leq \Delta t_{acc}$ becomes extremely small.  For example, when $p = 4$ (see Figure~\ref{fig:ex1cfl}, then $\Delta t = \mathcal{O}({\Delta x}^3)$ which usually can not be accepted in practice.  Of course, one can use other higher-order Runge-Kutta methods, such as the classical fourth order Runge-Kutta method  
\begin{equation}
  \label{eq:RK4}
  \begin{split}
    k^{(1)} & = \mathcal{L}_h(u_h^n), \\
    k^{(2)} & = \mathcal{L}_h(t_n + \frac{1}{2}\Delta t, u_h^n + \frac{1}{2}\Delta t k^{(1)}), \\
    k^{(3)} & = \mathcal{L}_h(t_n + \frac{1}{2}\Delta t, u_h^n + \frac{1}{2}\Delta t k^{(2)}), \\
    k^{(4)} & = \mathcal{L}_h(t_n + \Delta t, u_h^n + \Delta t k^{(3)}), \\
    u_h^{n+1} & = u_h^n + \frac{1}{6}\Delta t \left(k^{(1)} + 2 k^{(2)} + 2 k^{(3)} + k^{(4)}\right), 
  \end{split}
\end{equation}
the fourth order (ten stages) SSP Runge-Kutta method~\cite{Gottlieb:2009}, etc.  Using higher-order time discretizations slightly relieve the restriction~\eqref{eq:order} of $\Delta t$, but it does not solve the problem completely.  Due to the superconvergence property, the spatial order of accuracy increases fast with the polynomial degree $p$.  Therefore, using fixed order time integrators will eventually face the same dilemma: computationally expensive for small $p$ and not accurate enough for large $p$.  To solve this problem systematically, we have to investigate efficient time discretizations which satisfy the following basic properties: 
\begin{itemize}
  \item It can be constructed systematically for the arbitrary order of accuracy.
  \item It should have good stability property. 
  \item It needs to have an explicit scheme for hyperbolic equations with DG methods.  Further, it is expected to have an implicit scheme and preferred to have a semi-implicit (IMEX) scheme for other types of equations, such as convection-diffusion equations.
\end{itemize}
Besides above basic properties, there are also some advanced properties that considered as a plus,
\begin{itemize}
  \item Good strong stability preserving property for hyperbolic equations. 
  \item Suitable structure for parallel in time algorithms, such as Parareal~\cite{Lions:2001}, PFASST~\cite{Emmett:2012}, multigrid in time, etc.
  \item Good $hp$ adaptive structure to be coupled with DG methods.
\end{itemize}
These superior properties do not directly relate to the current development bottleneck of the post-processing technique.  However, those state-of-the-art developments of time discretization will play a role in the future research. 

\section{Efficient Time Discretizations}
To deal about the challenges mentioned in the previous section, we propose two state-of-the-art time discretization techniques in this section.  

\subsection{The SDG Method}
The main focus of this paper is the SDG method, which has been introduced recently by Li, Benedusi and Krause~\cite{Li:2017}. The SDG method is an iterative method derived from the DG method for time discretization, with the order of accuracy increased by one for each additional iteration. Respect to the post-processing technique, attractively, the SDG method can be constructed systematically for the arbitrary order of accuracy, and its explicit scheme has competitive stability compared existed time integrators. In the following, we first briefly present this method as described in~\cite{Li:2017}, and then provide a derivation of its explicit scheme from the point of view of a correction method.   

Consider the ODE system 
\begin{equation}
  \label{eq:ODEs}
  \begin{split}
    & u_t  = f(t,u(t)), \qquad t \in [0,T], \\
    & u(0)  = u_0,
  \end{split}
\end{equation}
where $u_0, u(t) \in \mathbb{R}^{n_s}$. 

We divide the time interval $[0, T]$ into $N$ subintervals by $0 = t_0 < t_1 < \cdots < t_n < \cdots < t_N = T$.  Let $I_n = [t_n, t_{n+1}]$, $\Delta t_n = t_{n+1} - t_n$, and $u_n$ denotes the numerical approximation of $u(t_n)$ with $u_0 = u(0)$.  

Then, we define the DG approximation space
\begin{equation}
  \label{eq:space}
    V^p_h = \left\{v_h: v_h\big\vert_{I_n} \in \mathbb{P}^p(I_n)\right\},
\end{equation}
where $\mathbb{P}^p$ denote the space of all polynomials of degree up to $p$.  The DG approximation, $u_h$, for solving~\eqref{eq:space} satisfies the weak formulation
\begin{equation}  
    \label{eq:weak}
    -\int_{t_n}^{t_{n+1}} u_h {(v_h)}_t \,dt + u_h(t_{n+1}^{-})v_h(t_{n+1}^{-}) - u_h(t_{n}^{- })v_h(t_{n}^{+}) = \int_{t_n}^{t_{n+1}} f(t,u_h(t))v_h \,dt, 
\end{equation}
for all $v_h \in V^p_h$,  here $u_h(t_n^{-}) = u_n$ is given from the previous time step.  On the interval $[t_n, t_{n+1}]$, $u_h$ is constructed in the nodal form that 
\begin{equation} 
  \label{eq:uh}
  u_h(t) = \sum\limits_{m = 0}^p u_{n,m}\ell_{n,m}(t), \qquad t \in [t_n, t_{n+1}],
\end{equation}
where $\left\{\ell_{n,m}(t)\right\}$ is the basis of Lagrange polynomials of degree $p$ with the right {\bf Gau\ss-Radau points} ${\left\{t_{n,m}\right\}}_{m=0\ldots p}$, so $u_{n+1} = u_{n,p}$.  By inserting the nodal representation~\eqref{eq:uh} into~\eqref{eq:weak}, we obtain, for $0 \leq j \leq p$, 
\[
  -\sum\limits_{m=0}^p u_{n,m}\int_{t_n}^{t_{n+1}} \ell_m(t)\ell'_j(t)\, dt + u_{n,p}\delta_{jp} - u_n\ell_j(t_n^{+}) = \int_{t_n}^{t_{n+1}}f(t,u_h(t))\ell_j(t)\,dt.
\]
By denoting $U = {[u_{n,0}, u_{n,1}, \ldots, u_{n,p}]}^T$, we can rewrite the above formula into a matrix form on the reference interval $[-1,1]$  
\[
  L U + \frac{\Delta t_n}{2}F(U) + B = 0,
\]
with  
\[
  L_{i,j} = \int_{-1}^{1} \ell'_i(t)\ell_j(t)\,dt - \delta_{ip}\delta_{jp}, \qquad B = {u_n\left[\ell_0(-1), \ell_1(-1), \ldots, \ell_p(-1)\right]}^T,
\]
and 
\[
  \begin{split}
    F(U) & = {\left[\int_{-1}^1 f\left(t, \sum\limits_{m=0}^p u_{n,m}\ell_m(t)\right)\ell_0(t)\,dt, \ldots, \int_{-1}^1 f\left(t, \sum\limits_{m=0}^p u_{n,m}\ell_m(t)\right)\ell_p(t)\,dt\right]}^T \\
    &\approx  {\left[\omega_0 f_{n,0}, \omega_1 f_{n,1}, \ldots, \omega_p f_{n,p}\right]}^T. 
  \end{split}
\]
where ${\left\{\omega_{m}\right\}}_{m=0\ldots p}$ be the corresponding quadrature weights of right Gau\ss-Radau points ${\left\{t_{n,m}\right\}}_{m=0\ldots p}$.

To present the SDG method, we introduce the coefficient matrix $\tilde{L} = \left\{\tilde{L}_{i,j}\right\} = L_\Delta L^{-1}$, where 
\[
    L_\Delta = \left[\begin{array}{ccccc}
            -1  &    & & & \\
            1  & -1 & & &\\
               &\ddots & \ddots & &\\
               & & 1 & -1 & \\
               &&  & 1 & -1
    \end{array}\right],
\]
is an approximation of $L$ (see~\cite{Li:2017}).  Finally, we give the algorithm of the explicit SDG method in Algorithm~\ref{algorithm:SDG}. 

\begin{algorithm}
  \caption{Explicit SDG Scheme, $ExSDG_p^{K}$.} \label{algorithm:SDG}
  \DontPrintSemicolon
  \KwData{$u_n$ from the previous time step.}
  \KwResult{Solution $u_{n+1}$ for the new time step.}
  \hfill\\
  \tcp{Compute the initial approximation with the explicit Euler method}
  $u_{n,0}^0 = u_{n} + \left(t_{n,0} - t_n\right) f(t_n, u_n)$\;
  \For{$m = 0 \ldots p-1$}{
    $u_{n,m+1}^0 = u_{n,m}^0 + \left(t_{n,m+1} - t_{n,m}\right) f(t_{n,m}, u^0_{n,m})$;
  }
  \hfill\\
  \tcp{Compute successive approximation}
  \For{$k = 1 \ldots K-1$}{
    $u^{k+1}_{n,0}  =  u_n + \frac{\Delta t_n}{2}\sum\limits_{j=0}^p {\tilde{L}}_{0,j} \omega_j  f(t_{n,j}, u^{k}_{n,j})$\;
    \For{$m = 0 \ldots p-1$}{
      $u^{k+1}_{n,m+1}  = u^{k+1}_{n,m} + \frac{\Delta t_n}{2}\omega_{m}\left(f(t_{n,m},u^{k+1}_{n,m}) - f(t_{n,m}, u^{k}_{n,m})\right)$ \; $\qquad\qquad\qquad\qquad + \frac{\Delta t_n}{2} \sum\limits_{j=0}^p \tilde{L}_{m+1,j}\omega_j f(t_{n,j},u^{k}_{n,j})$;
    }
  }
  $u_{n+1} = u_{n,p}^K$\;
\end{algorithm}

\subsubsection{Correction Method Based on Weak Formulation}
As an appendix of Algorithm~\ref{algorithm:SDG}, we present a derivation for the explicit SDG scheme from a correction method based on the DG weak formulation.

Denote $u^k \in V_h^p$ the $k$-th approximation for the DG solution $u_h$, the defect equation is given by
\[
  \delta^k = u_h - u^k \in V_h^p,
\]
where
\[ 
  u^k(t) = \sum\limits_{m = 0}^p u^k_{n,m}\ell_{n,m}(t), \quad \delta^k(t) = \sum\limits_{m = 0}^p \delta^k_{n,m}\ell_{n,m}(t), \qquad t \in [t_n, t_{n+1}].
\]
Substituting $u_h = u^k + \delta_h$ into the weak form~\eqref{eq:weak}, we have
\begin{align*} 
    &\int_{t_n}^{t_{n+1}} \delta^k {(v_h)}_t \,dt - \delta^k(t_{n+1}^{-})v_h(t_{n+1}^{-}) +  \int_{t_n}^{t_{n+1}} \left(f(t,u^k + \delta^k) - f(t,u^k)\right)v_h \,dt \\
    =\, & -\int_{t_n}^{t_{n+1}} u^k {(v_h)}_t \,dt + u^k(t_{n+1}^{-})v_h(t_{n+1}^{-}) - u_h(t_{n}^{- })v_h(t_{n}^{+}) - \int_{t_n}^{t_{n+1}}f(t, u^k)v_h \,dt.
\end{align*}
By denote 
\[
  U^k = {[u^k_{n,0}, u^k_{n,1}, \ldots, u^k_{n,p}]}^T, \quad \boldsymbol{\delta}^k = {[\delta^k_{n,0}, \delta^k_{n,1}, \ldots, \delta^k_{n,p}]}^T
\]
we rewrite it in the matrix form
\begin{equation}
  \label{eq:weak-delta}
  L^{-1}\left(L\boldsymbol{\delta}^k + \frac{\Delta{t_n}}{2}\left(F(U^k+\boldsymbol{\delta}^k) - F(U^k)\right)\right) = - U^k - \frac{\Delta{t_n}}{2}L^{-1}F(U^k) - L^{-1}B.
\end{equation}
So far, we note that the difficulty of getting $\boldsymbol{\delta}^k$ from~\eqref{eq:weak-delta} is the same as solving the weak formulation~\eqref{eq:weak}.  To obtain a simple and explicit scheme, we simplify the left-hand side of~\eqref{eq:weak-delta} by approximating $\boldsymbol{\delta}^k$ in the first order.  First, we replace the $L$ in the left-hand side of~\eqref{eq:weak-delta} with $L_\Delta$, then~\eqref{eq:weak-delta} becomes
\[
  L_\Delta\boldsymbol{\delta}^k + \frac{\Delta{t_n}}{2}\left(F(U^k+\boldsymbol{\delta}^k) - F(U^k)\right) = -L_\Delta U^k - \frac{\Delta t_n}{2}\tilde{L}F(U^k) - \tilde{L}B.
\]
Further, we approximate $\delta^k$ by a piecewise constant representation that 
\begin{equation}
  \label{eq:delta1}
  \delta^k(t) \approx \sum\limits_{m=0}^{p} \delta_{n,m}\chi[t_{n,m-1},t_{n,m}), \quad t \in [t_n, t_{n+1}],
\end{equation}
where $t_{n,-1} = t_n$ and $\chi$ is the standard characteristic function.  For the nonlinear term 
\begin{equation}
  \label{eq:non1}
  \int_{t_n}^{t_{n+1}} \left(f(t,u^k + \delta^k) - f(t,u^k)\right)v_h \,dt,
\end{equation}
to generate an explicit scheme, we choose the test function $v_h$ to be the basis of Lagrange polynomials of degree $p$ with points ${\left\{t_{n,m}\right\}}_{m=-1\ldots p-1}$.  Then, by using the Gau\ss-Radau quadrature, we have 
\[
  F(U^k+\boldsymbol{\delta}^k) \approx F_\Delta(U^k+\boldsymbol{\delta}^k) = {\left[0, \omega_0 f(t_{n,0}, u_{n,0}^k+\delta_{n,0}^k), \ldots, \omega_{p-1}f(t_{n,p-1}, u_{n,p-1}^k+\delta_{n,p-1}^k)\right]}^T,
\]
and
\[
  F(U^k) \approx F_\Delta(U^k) =  {\left[0, \omega_0 f(t_{n,0}, u_{n,0}^k), \ldots, \omega_{p-1}f(t_{n,p-1}, u_{n,p-1}^k)\right]}^T.
\]
Finally, we update the new approximation $U^{k+1} = U^k+\boldsymbol{\delta}^k$, that 
\[
  -L_\Delta U^{k+1} = \frac{\Delta{t_n}}{2}\left(F_\Delta(U^{k+1}) - F_\Delta(U^k)\right) + \frac{\Delta t_n}{2}\tilde{L}F(U^k) + \tilde{L}B.
\]
By writing the above formula component-wise, we have 
\[ 
  \begin{split}
    u^{k+1}_{n,0} & =  u_n + \frac{\Delta t_n}{2}\sum\limits_{j=0}^p {\tilde{L}}_{0,j} \omega_j  f(t_{n,j}, u^{k}_{n,j}) \\
    u^{k+1}_{n,m+1} & = u^{k+1}_{n,m} + \frac{\Delta t_n}{2}\omega_{m}\left(f(t_{n,m},u^{k+1}_{n,m}) - f(t_{n,m}, u^{k}_{n,m})\right) \\
    &\qquad\qquad + \frac{\Delta t_n}{2} \sum\limits_{j=0}^p \tilde{L}_{m+1,j}\omega_j f(t_{n,j}, u^{k}_{n,j}), \qquad 0 \leq m \leq p-1.
  \end{split}
\] 
It is the explicit SDG scheme as shown in Algorithm~\ref{algorithm:SDG}.
\begin{rem}
  If we choose the same basis function $\left\{\ell_{n,m}(t)\right\}$ to approximate term~\eqref{eq:non1}, then the implicit SDG scheme will be obtained.  One can also choose the piecewise constant basis $\left\{\chi [t_{n,m-1},t_{n,m})\right\}$ to get the same result.  In fact, term~\eqref{eq:non1} can be approximated in various ways, here we simply choose the one leads to the same scheme shown in Algorithm~\ref{algorithm:SDG}.  
\end{rem}

A more detailed introduction, such as the stability analysis, the multigrid algorithm of the SDG method, can be found in~\cite{Li:2017}.  Here, we only present the error estimate of the explicit SDG method as follows. 
\begin{thm}\label{thm:SDG}   
  The explicit SDG methods with $K$ time of iterations are $\min\{2p+1, K+1\}$ order accurate methods for the ODE system~\eqref{eq:ODEs}. 
\end{thm}
\begin{proof}
  See \cite{Li:2017}.
\end{proof}
Attractively, Theorem~\ref{thm:SDG} shows the upper bound of the accuracy order of the SDG method is $2p+1$ instead of $p+1$, the general accuracy order for methods constructed with polynomials of degree $p$ (or $p+1$ nodes). This property can be referred as superconvergence in time, which is directly inherited from applying the DG method for solving ODEs, see~\cite{Adjerid:2002}. Once the SDG method is coupled with the DG spatial discretization and post-processing technique, a space-time superconvergence approximation is directly obtained. In addition, the SDG method has many other attractive features, such as semi-implicit schemes, the multigrid algorithm in time, a suitable structure for parallel in time algorithms, etc. These superior features will be discussed in the future works.    
\begin{rem}
  We note that using Gau\ss-Radau nodes in the nodal form~\eqref{eq:uh} is crucial for Theorem~\ref{thm:SDG}.  Other choices of nodal points will normally cause the reduction of the accuracy order.  For example, Gau\ss-Lobatto nodes lead to the accuracy order of $2p$ and uniform nodes lead to the accuracy order of $p+1$. 
\end{rem}

\subsection{The Spectral Deferred Correction Method}
As mentioned earlier, the SDG method introduced in the previous section is a latest research result. Although the SDG method has demonstrated its excellent properties and potentials, the related researches such as the application of the semi-discretized system, parallel algorithm, efficient implementation, etc.\ are still under development. Therefore, as an alternative, we briefly introduce another candidate: the spectral deferred correction (SDC) method introduced by Dutt, Greengard, and Rokhlin~\cite{Dutt:2000} in 2000. The SDC method is based on the integral form of~\eqref{eq:ODEs}, and corrected iteratively with low-order time integration methods. Similar to the SDG method, the order of accuracy increased by one for each additional iteration. As a method introduced for more than a decade, there are many previous studies of the SDC method, such as the semi-implicit SDC method~\cite{Minion:2003}, parallel in time algorithms~\cite{Emmett:2012}, the multi-level SDC method~\cite{Speck:2015}, etc. Furthermore, there are already some applications of the SDC method respected to DG methods. For example, the SSP property of the SDC method has been studied for the low-order case~\cite{Liu:2008}, and the SDC method has shown to be an efficient time discretization for local discontinuous Galerkin methods~\cite{Xia:2007}.

We first rewrite ODEs~\eqref{eq:ODEs} into in the integral form of the interval $[t_n, t_{n+1}]$:
\begin{equation}
  \label{eq:integral}
  u(t_{n+1}) = u(t_n) + \int_{t_n}^{t_{n+1}}f(\xi, u(\xi))\,d\xi.
\end{equation}
Then we divide the interval $[t_n, t_{n+1}]$ by choosing the Gau\ss-Radau nodes ${\left\{t_{n,m}\right\}}_{m=0\ldots p}$ such that $t_n < t_{n,0} < t_{n,1} < \cdots < t_{n,p} = t_{n+1}$, and $\Delta t_{n,m} = t_{n,m+1} - t_{n,m}$. We note that the Gau\ss-Radau nodes is chosen to create an easy comparison to the SDG method (computational cost, accuracy order, etc.), one can also use the Gau\ss-Lobatto or Chebyshev nodes. Start with $u_n$, we give the algorithm of the explicit SDC method in Algorithm~\ref{algorithm:SDC}.   
\begin{algorithm}
  \caption{Explicit SDC Scheme, $ExSDC_p^{K}$.} \label{algorithm:SDC}
  \DontPrintSemicolon
  \KwData{$u_n$ from the previous time step.}
  \KwResult{Solution $u_{n+1}$ for the new time step.}
  \hfill\\
  \tcp{Compute the initial approximation with the explicit Euler method}
  $u_{n,0}^0 = u_{n} + \left(t_{n,0} - t_n\right) f(t_n, u_n)$\;
  \For{$m = 0 \ldots p-1$}{
    $u_{n,m+1}^0 = u_{n,m}^0 + \Delta t_{n,m} f(t_{n,m}, u^0_{n,m})$;
  }
  \hfill\\
  \tcp{Compute successive approximation}
  \For{$k = 1 \ldots K-1$}{
    $u^{k+1}_{n,0}  =  u_n$\;
    \For{$m = 0 \ldots p-1$}{
      $u^{k+1}_{n,m+1}  = u^{k+1}_{n,m} + \Delta t_{n,m}\left(f(t_{n,m},u^{k+1}_{n,m}) - f(t_{n,m}, u^{k}_{n,m})\right)$\;$\qquad\qquad\qquad\qquad + \frac{\Delta t_n}{2} \sum\limits_{j=0}^p S_{m+1,j} f(t_{n,j},u^{k}_{n,j})$;
    }
  }
  $u_{n+1} = u_{n,p}^K$\;
\end{algorithm}
In Algorithm~\ref{algorithm:SDC}, the coefficients $S_{m+1,j}$ are given by 
\[
  S_{m+1,j} = \frac{2}{\Delta t_n}\int_{t_{n,m}}^{t_{n,m+1}} \ell_{j}(t)\, dt,
\]
where $\left\{\ell_{j}(t)\right\}$ is the basis of Lagrange polynomials of degree $p$ with the nodes ${\left\{t_{n,m}\right\}}_{m=0\ldots p}$, see~\cite{Xia:2007} for details. 

\begin{thm}\label{thm:SDC}
  The explicit SDC methods with $K$ time of iterations are $\min\{2p+1, K+1\}$ order accurate methods for the ODE system~\eqref{eq:ODEs}. 
\end{thm}
\begin{proof}
  See~\cite{Hagstrom:2006}.
\end{proof}

\begin{rem}
  We note that with same the polynomial degree $p$ and the number of iteration $K$, the SDG method (Algorithm~\ref{algorithm:SDG}) and the SDC method (Algorithm~\ref{algorithm:SDC}) have the same computational cost and accuracy order. Therefore, the two methods have very similar behaviors in the primary numerical examples in the next section. The reason that we focus on the SDG method is that the SDG method is directly derived from the DG method and inherited its superconvergence property. Furthermore, the SDG method can be treated as a natural extension for the space-time DG formulation.  
\end{rem}

\subsection{Adaptive Strategy}
Adaptive strategies play an important role in practical applications.  Both the SDG and SDC methods are good candidates for the adaptive implementation.  Especially, the SDG method can take the similar $hp$ strategies for the DG method and can be easily apopted in a space-time adaptive framework.   

In this paper, respect to the post-processing technique, we consider a simple order adaptive strategy: adaptive the number of iterations, $K$.  From Theorem~\ref{thm:SDG} and Theorem~\ref{thm:SDC}, we see that if $K=2p$ both the SDG and SDC methods can achieve the accuracy order of $2p+1$ in time same as the post-processed approximation in space.   However, due to the stability (CFL) condition considering, the explicit numerical scheme usually satisfies $\Delta t < \Delta x$, it is may not be necessary to use as many as $2p$ iterations.  Instead, one can adapt the number of iterations according to different problems.  The error indicator can be simply defined as 
\begin{equation}
  \label{eq:adapt}
  \left|u_h^{K}(t_{n+1}^{-}) - u_h^{K-1}(t_{n+1}^{-})\right| = \left|u_{n,p}^K - u_{n,p}^{K-1}\right| < \varepsilon,
\end{equation}
where $\varepsilon$ is the tolerant error.  The theoretical analysis of how to choose the number of iterations (or say $\varepsilon$) requires a detailed analysis of the fully discretized schemes which will be studied in the forthcoming work. 
\begin{rem}
   We note that the error indicator~\eqref{eq:adapt} presented here is mainly for easy demonstration.  For more complicated problems, further posteriori error estimates need to be investigated to construct a better error indicator. 
\end{rem}

\section{Numerical Results}
In this section, we perform numerical experiments of the DG scheme coupled with different time integration methods (SDG and SDC) to linear and nonlinear equations for the post-processing technique.  We note that in this section, the quadruple precision is used to avoid round-off errors.  Also, all problems are computed on uniform meshes with periodic boundary conditions to prevent disturbances from spatial discretization. 

\subsection{Linear Hyperbolic Equation}
First, we consider the linear hyperbolic equation
\begin{equation}
  \label{eq:linear}
  \begin{split}
    u_t + u_x & = 0, \quad (x,t) \in [0,1]\times [0, T],\\
       u(x,0) & =  \sin(2\pi x),
  \end{split}
\end{equation}
with final time $T = 1$ over uniform meshes. For spatial discretization with polynomials of degree $p$, we use the $ExSDG_p^{2p}$ and $ExSDC_{p}^{2p}$ for discretizing in time. In this example, the CFL number is chosen to be $0.1$ for all the case ($\Delta t = 0.1\Delta x$). The $L^2$ norm error and the respect accuracy order are given in Table~\ref{table:linear}. Here, since both the SDG and SDC method lead to the same results (spatial error dominate), we do not distinguish them in Table~\ref{table:linear} and the rest of this section. In table~\ref{table:linear}, we can see that by using the purposed SDG/SDC method, the spatial accuracy shows the desired superconvergence order of $2p+1$ ( $2p+2$ numerically). In Figure~\ref{fig:Time}, for spatial element number $N = 160$, we compare the computational time cost of using the third-order Runge-Kutta~\eqref{eq:RK3} and the SDG/SDC methods when the same accuracy is achieved (see Table~\ref{table:linear}). In addition, the computational cost of the adaptive SDG/SDC method is also presented in Figure~\ref{fig:Time}, and this adaptive strategy can save up 20\% to 40\% of the computational cost. We can see that the Runge-Kutta method is more efficient than the SDG/SDC method only for $p = 1$ case, once $p > 1$ the SDG method starts to demonstrate its advantage. Furthermore, we provide the Table~\ref{table:ratio} to show the ratio of the time cost for both methods, which shows that the SDG method is significantly more efficient than the third-order Runge-Kutta method when $p$ is large.  
\begin{rem}
  We only present the comparison of computational time cost once to save space, since the ratio of the time cost between the Runge-Kutta and the SDG/SDC methods remains the same for each time step.         
\end{rem}

\begin{table}\centering 
\begin{tabular}{@{}cccccccc@{}}\toprule
  & & & \multicolumn{2}{c}{DG error} & \phantom{abc}& \multicolumn{2}{c}{After post-processing} \\ 
\cmidrule{4-5} \cmidrule{7-8}
Degree & $N$ & & $L^2$ error & Order && $L^2$ error & Order \\ 
\midrule
$p = 2$ & 20  && 1.07E-04 &   --  && 4.10E-06 &   -- \\
        & 40  && 1.34E-05 &  3.00 && 9.42E-08 &  5.44\\
        & 80  && 1.67E-06 &  3.00 && 2.40E-09 &  5.30\\
        & 160 && 2.09E-07 &  3.00 && 6.63E-11 &  5.18\\
\midrule
$p = 3$ & 20  && 2.06E-06 &   --  && 6.97E-08 &   --\\ 
        & 40  && 1.29E-07 &  4.00 && 2.82E-10 &  7.95\\ 
        & 80  && 8.07E-09 &  4.00 && 1.14E-12 &  7.95\\ 
        & 160 && 5.04E-10 &  4.00 && 4.67E-15 &  7.93\\ 
\midrule
$p = 4$ & 20  && 3.19E-08 &   --  && 2.19E-09 &   -- \\
        & 40  && 1.00E-09 &  4.99 && 2.20E-12 &  9.96\\
        & 80  && 3.14E-11 &  5.00 && 2.17E-15 &  9.99\\
        & 160 && 9.82E-13 &  5.00 && 2.12E-18 & 10.00\\
\midrule
\bottomrule
\end{tabular}
\caption{$L^2$ errors for the DG approximation together with the approximation after post-processing for the linear hyperbolic equation~\eqref{eq:linear}.}
\label{table:linear}
\end{table}

    

\begin{figure}[!ht]
    \centering
    \begin{minipage}[c]{0.04\textwidth}
        \begin{sideways}
          \centerline{$\,$}
        \end{sideways}
    \end{minipage}
    \begin{minipage}[c]{0.46\textwidth}
      \centerline{$\qquad\quad \text{DG error}$}
    \end{minipage}
    \begin{minipage}[c]{0.46\textwidth}
      \centerline{$\qquad\quad \text{After post-processing}$}
    \end{minipage}
    \begin{minipage}[c]{0.04\textwidth}
        \begin{sideways}
          \centerline{$p = 3$}
        \end{sideways}
    \end{minipage}
    \begin{minipage}[c]{0.46\textwidth}
    \includegraphics[width=1.\textwidth]{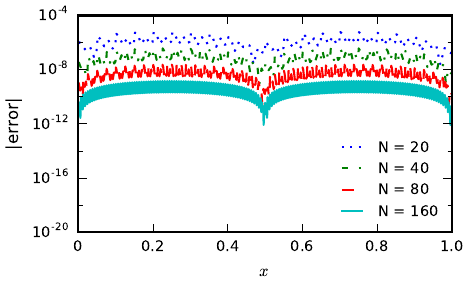}
    \end{minipage}
    \begin{minipage}[c]{0.46\textwidth}
    \includegraphics[width=1.\textwidth]{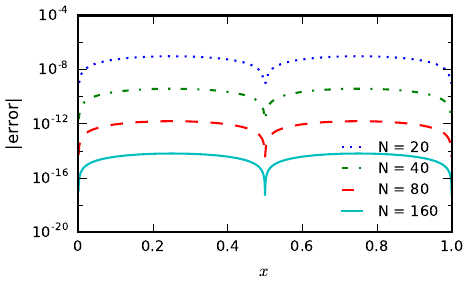}
    \end{minipage}
    
    \begin{minipage}[c]{0.04\textwidth}
        \begin{sideways}
          \centerline{$p = 4$}
        \end{sideways}
    \end{minipage}
    \begin{minipage}[c]{0.46\textwidth}
    \includegraphics[width=1.\textwidth]{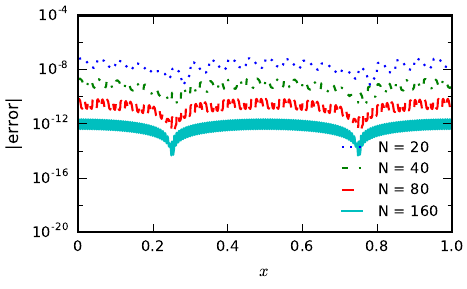}
    \end{minipage}
    \begin{minipage}[c]{0.46\textwidth}
    \includegraphics[width=1.\textwidth]{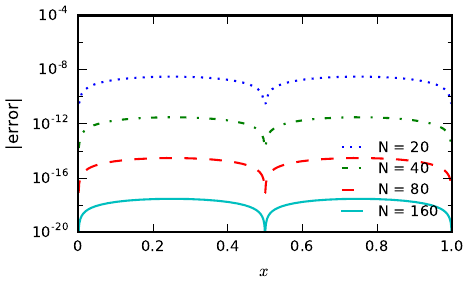}
    \end{minipage}

    \begin{center}
    \caption{\label{fig:linear}
        Comparison of the point-wise errors of the DG solution and the post-processed solution for the linear hyperbolic equation~\eqref{eq:linear}.
}
     \end{center}
\end{figure}

\begin{figure}[!ht]
    \centering
    \includegraphics[width=.7\textwidth]{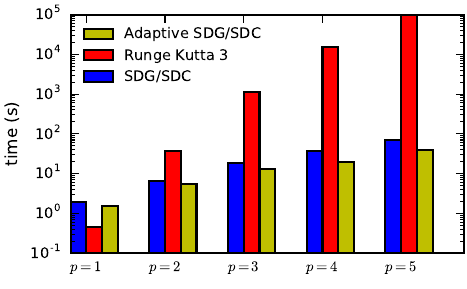}
    
    \begin{center}
    \caption{\label{fig:Time}
        Comparison of the computational time of using the third order Runge-Kutta method, the SDG/SDC method and the adaptive SDG/SDC method when the desired accuracy is achieved (spatial error dominate).  For the linear hyperbolic equation~\eqref{eq:linear}, with $160$ elements in space. 
}
     \end{center}
\end{figure}

\begin{table}\centering 
\begin{tabular}{@{}ccccccc@{}}\toprule
  time (s) && $\mathbb{P}^1$ & $\mathbb{P}^2$ & $\mathbb{P}^3$ & $\mathbb{P}^4$ & $\mathbb{P}^5$\\ 
\midrule
Runge-Kutta && 4.60E-01  & 3.73E+01 & 1.12E+03 & 1.51E+04 & 2.00E+05\\
SDG/SDC     && 1.93E+00  & 6.56E+00 & 1.83E+01 & 3.67E+01 & 6.94E+01\\
Ratio       && 0.24      &   5.7    &  61      & 411      &  2882   \\
\midrule
\bottomrule
\end{tabular}
\caption{The ratio of the computational time (s) of using the third order Runge-Kutta method, the SDG/SDC method and the adaptive SDG/SDC method when the desired accuracy is achieved (spatial error dominate).  For the linear hyperbolic equation~\eqref{eq:linear}, with $160$ elements in space.}
\label{table:ratio}
\end{table}

    

\subsection{Variable Coefficient Equation}
Then, we considered a variable coefficient equation 
\begin{equation}
  \label{eq:var}
  \begin{split}
    u_t + (au)_x & = 0, \quad (x,t) \in [0,1]\times [0, T],\\
       u(x,0) & =  \sin(2\pi x),
  \end{split}
\end{equation}
where the variable coefficient $a(x,t) = 2 + \sin(2\pi(x+t))$ and the right-side term $f(x,t)$ is chosen to make the exact solution be $u(x,t) = \sin(2\pi(x-t)))$. Equation~\eqref{eq:var} is also a linear hyperbolic equation, similar to the constant coefficient case~\eqref{eq:linear}, we compute the approximation with final time $T = 1$ over uniform meshes. In this example, due to its variable coefficient $a(x,t)$ the computation cost for each time step is significantly higher than the constant coefficient case. Therefore, time integrators that allows large time step size is more appreciated. Again, both the SDG ($ExSDG_p^{2p}$) and the SDC ($ExSDC_{p}^{2p}$) methods have been tested for time discretization with the CFL number $0.05$ for all polynomial degrees. Table~\ref{table:var} shows the $L^2$ error and the respective accuracy order for both the original DG solutions and the post-processed approximations. In Table~\ref{table:var}, we can see that both proposed time integrators work well, the $L^2$ errors are dominated by the spatial error of superconvergence order of $2p+1$ is observed for the post-processed approximations.        

\begin{table}\centering 
\begin{tabular}{@{}cccccccc@{}}\toprule
  & & & \multicolumn{2}{c}{DG error} & \phantom{abc}& \multicolumn{2}{c}{After post-processing} \\ 
\cmidrule{4-5} \cmidrule{7-8}
Degree & $N$ & & $L^2$ error & Order && $L^2$ error & Order \\ 
\midrule
$p = 2$ & 20  && 1.07E-04 &   --  && 1.90E-06 &   -- \\
        & 40  && 1.34E-05 &  3.00 && 2.86E-08 &  6.05\\
        & 80  && 1.67E-06 &  3.00 && 5.95E-10 &  5.59\\
        & 160 && 2.09E-07 &  3.00 && 1.88E-11 &  4.99\\
\midrule
$p = 3$ & 20  && 2.06E-06 &   --  && 6.87E-08 &   --\\ 
        & 40  && 1.29E-07 &  4.00 && 2.74E-10 &  7.97\\ 
        & 80  && 8.07E-09 &  4.00 && 1.07E-12 &  8.00\\ 
        & 160 && 5.04E-10 &  4.00 && 4.16E-15 &  8.01\\ 
\midrule
$p = 4$ & 20  && 3.22E-08 &   --  && 2.19E-09 &   -- \\
        & 40  && 1.01E-09 &  5.00 && 2.21E-12 &  9.95\\
        & 80  && 3.14E-11 &  5.00 && 2.19E-15 &  9.98\\
        & 160 && 9.82E-13 &  5.00 && 2.19E-18 &  9.97\\
\midrule
\bottomrule
\end{tabular}
\caption{$L^2$ errors for the DG approximation together with the approximation after post-processing for the variable coefficient equation~\eqref{eq:var}.}
\label{table:var}
\end{table}

    

\begin{figure}[!ht]
    \centering
    \begin{minipage}[c]{0.04\textwidth}
        \begin{sideways}
          \centerline{$\,$}
        \end{sideways}
    \end{minipage}
    \begin{minipage}[c]{0.46\textwidth}
      \centerline{$\qquad\quad \text{DG error}$}
    \end{minipage}
    \begin{minipage}[c]{0.46\textwidth}
      \centerline{$\qquad\quad \text{After post-processing}$}
    \end{minipage}
    \begin{minipage}[c]{0.04\textwidth}
        \begin{sideways}
          \centerline{$p = 3$}
        \end{sideways}
    \end{minipage}
    \begin{minipage}[c]{0.46\textwidth}
    \includegraphics[width=1.\textwidth]{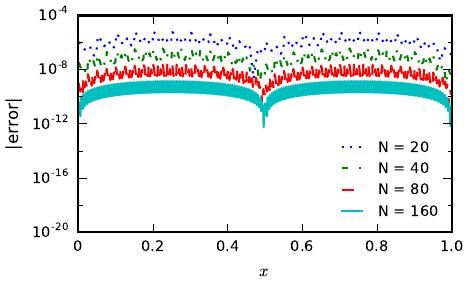}
    \end{minipage}
    \begin{minipage}[c]{0.46\textwidth}
    \includegraphics[width=1.\textwidth]{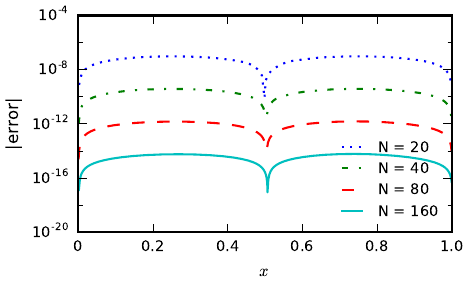}
    \end{minipage}
    
    \begin{minipage}[c]{0.04\textwidth}
        \begin{sideways}
          \centerline{$p = 4$}
        \end{sideways}
    \end{minipage}
    \begin{minipage}[c]{0.46\textwidth}
    \includegraphics[width=1.\textwidth]{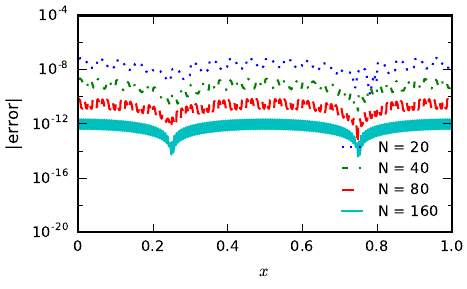}
    \end{minipage}
    \begin{minipage}[c]{0.46\textwidth}
    \includegraphics[width=1.\textwidth]{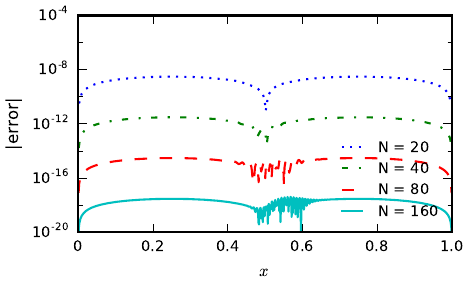}
    \end{minipage}

    \begin{center}
    \caption{\label{fig:var}
        Comparison of the point-wise errors of the DG solution and the post-processed solution for the variable coefficient equation~\eqref{eq:var}.
}
     \end{center}
\end{figure}

\subsection{Burgers Equation}
For the final example, we consider a nonlinear conservation law, Burgers equation, 
\begin{equation}
  \label{eq:burgers}
  \begin{split}
    u_t + \left(\frac{u^2}{2}\right)_x & = 0, \quad (x,t) \in [0,2\pi]\times [0, T],\\
       u(x,0) & =  \sin(x),
  \end{split}
\end{equation}
with final time $T = 0.5$ over uniform meshes. Note that the final time is chosen before the appearance of the shock to avoid interruptions. For dealing the shocks with the post-processing technique, one can refer~\cite{Li:2015} for details. The $L^2$ error and the respective accuracy order for both the original DG solutions and the post-processed approximations are presented in Table~\ref{table:burgers}. We note that the numerical results show the post-processed approximation has the superconvergence order of $2p+1$. However, the theoretical analysis towards nonlinear conservation laws is still an ongoing work, see~\cite{Meng:2016}. At the moment of writing, there is no theoretical proof to guarantee the superconvergence order of $2p+1$ for nonlinear conservation laws.  

\begin{table}\centering 
\begin{tabular}{@{}cccccccc@{}}\toprule
  & & & \multicolumn{2}{c}{DG error} & \phantom{abc}& \multicolumn{2}{c}{After post-processing} \\ 
\cmidrule{4-5} \cmidrule{7-8}
Degree & $N$ & & $L^2$ error & Order && $L^2$ error & Order \\ 
\midrule
$p = 2$ & 20  && 3.36E-04 &   --  && 9.22E-04 &   --\\ 
        & 40  && 4.79E-05 &  2.81 && 3.57E-05 &  4.69\\ 
        & 80  && 6.57E-06 &  2.87 && 7.87E-07 &  5.50\\ 
        & 160 && 8.83E-07 &  2.89 && 1.39E-08 &  5.82\\ 
\midrule
$p = 3$ & 20  && 3.99E-05 &   --  && 8.07E-04 &   --\\ 
        & 40  && 2.55E-06 &  3.97 && 2.06E-05 &  5.30\\ 
        & 80  && 1.75E-07 &  3.86 && 1.96E-07 &  6.72\\ 
        & 160 && 1.16E-08 &  3.91 && 1.04E-09 &  7.56\\ 
\midrule
$p = 4$ & 20  && 1.72E-06 &   --  && 7.63E-04 &   --\\
        & 40  && 1.28E-07 &  3.74 && 1.48E-05 &  5.69\\
        & 80  && 4.35E-09 &  4.88 && 7.31E-08 &  7.66\\
        & 160 && 1.46E-10 &  4.89 && 1.34E-10 &  9.09\\
\midrule
\bottomrule
\end{tabular}
\caption{$L^2$ errors for the DG approximation together with the approximation after post-processing for the Burgers equation~\eqref{eq:burgers}.}
\label{table:burgers}
\end{table}

    

\begin{figure}[!ht]
    \centering
    \begin{minipage}[c]{0.04\textwidth}
        \begin{sideways}
          \centerline{$\,$}
        \end{sideways}
    \end{minipage}
    \begin{minipage}[c]{0.46\textwidth}
      \centerline{$\qquad\quad \text{DG error}$}
    \end{minipage}
    \begin{minipage}[c]{0.46\textwidth}
      \centerline{$\qquad\quad \text{After post-processing}$}
    \end{minipage}
    \begin{minipage}[c]{0.04\textwidth}
        \begin{sideways}
          \centerline{$p = 3$}
        \end{sideways}
    \end{minipage}
    \begin{minipage}[c]{0.46\textwidth}
    \includegraphics[width=1.\textwidth]{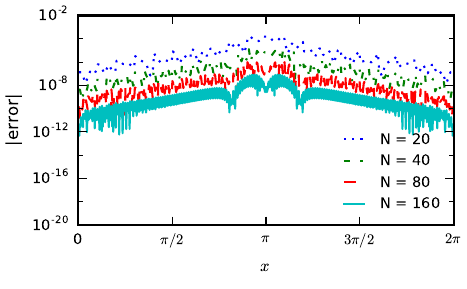}
    \end{minipage}
    \begin{minipage}[c]{0.46\textwidth}
    \includegraphics[width=1.\textwidth]{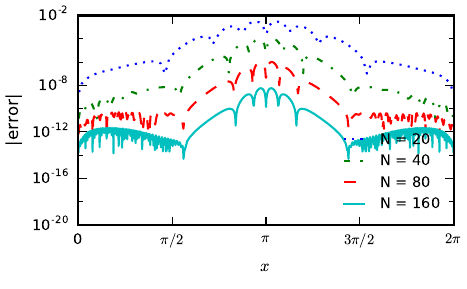}
    \end{minipage}
    
    \begin{minipage}[c]{0.04\textwidth}
        \begin{sideways}
          \centerline{$p = 4$}
        \end{sideways}
    \end{minipage}
    \begin{minipage}[c]{0.46\textwidth}
    \includegraphics[width=1.\textwidth]{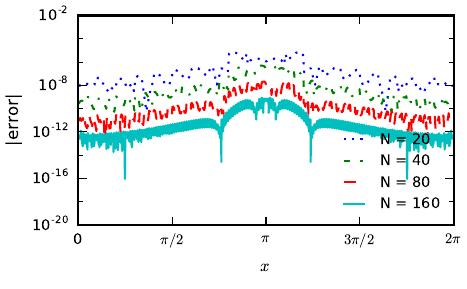}
    \end{minipage}
    \begin{minipage}[c]{0.46\textwidth}
    \includegraphics[width=1.\textwidth]{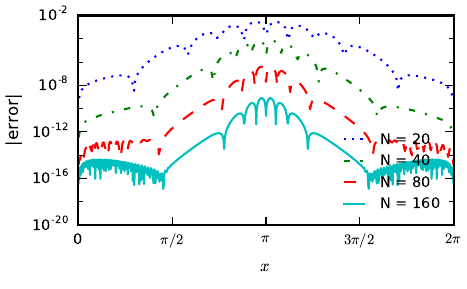}
    \end{minipage}

    \begin{center}
    \caption{\label{fig:burgers}
        Comparison of the point-wise errors of the DG solution and the post-processed solution for the Burgers equation~\eqref{eq:burgers}.
}
     \end{center}
\end{figure}

\section{Conclusion}
In this paper, we have studied two time discretizations for exploring the spatial superconvergence of DG methods (post-processing technique).  The focus is the SDG method, which is an iterative method derived from the DG method and demonstrates a superconvergence property in time.  As another choice, we also discussed the SDC method.  Numerical experiments are performed to verify that both methods are efficient for discretizing the semi-discretized system of DG methods with the post-processing technique.  Furthermore, the comparisons with the TVD explicit third Runge-Kutta methods are presented, and the enhancement of the efficiency is significant.  In particular, the SDG method is inherited from the DG method, it has a good application potential for discretizing in time for the DG method. In future work, we will explore different properties of the SDG method when discretizing the DG scheme in time, such as strong stability preserving property, the error estimates of the fully discretized system, parallel-in-time algorithms, etc.  

\section*{Acknowledgements}
Research supported by NSFC Grant 11801062, and NSFSC Grant 22NSFSC3834.

\bibliography{ref}
\bibliographystyle{plain}
\end{document}